%% file: main.tex
\documentclass[article, 12pt]{amsart}

\input{preamble}

\title{On intersections of locally quasiconvex subgroups}
\author{Nir Lazarovich}
\address{Department of Mathematics\\
 Technion\\
 Haifa, 32000 Israel}
 \email{lazarovich@technion.ac.il}
 \thanks{Lazarovich was supported by the Israel Science Foundation (grant no. 1576/23)}

\author{Zachary Munro}
\address{School of Mathematics, Fry Building, University of Bristol, United Kingdom}
\email{zachy.munro@bristol.ac.uk}
\thanks{Munro was supported by the National Science Foundation under Award No. DMS–
2503331.}
\date{}

\begin{document}

\maketitle

\begin{abstract}
Let $H_1, H_2$ be locally quasiconvex, torsion-free subgroups of a hyperbolic group $G$.
We prove $\rank(H_1\cap H_2)$ can be bounded in terms of $\rank(H_1)$ and $\rank(H_2)$.
\end{abstract}

\section{Introduction}

A group has the \emph{finitely-generated intersection property} (f.g.i.p.) if every intersection of finitely-generated subgroups is finitely-generated. 
The f.g.i.p. is also referred to as \emph{Howson's property} after the mathematician who first established the property in free groups \cite{howson1954intersection}.
He proved that if $H_1,H_2<F$ are subgroups of a free groups of ranks $r_1,r_2\ge 1$, then
$\rank(H_1\cap H_2)\le 2r_1r_2-r_1-r_2+1$.
Neumann \cite{neumann1956hanna} improved this bound (see also \cite{neumann1990walter}), 
and Friedman\cite{friedman2015sheaves} and Mineyev\cite{mineyev2012submultiplicativity}, in solving the Hanna Neumann Conjecture, obtained the sharp bound $\rank(H_1\cap H_2)\le r_1r_2-r_1-r_2+2$.
For surface groups, similar bounds for finitely-generated subgroups of a surface groups were obtained in \cite{antolin2022hanna, soma1990intersection, soma1991intersection}.

This article is concerned with quantified rank bounds for intersections of subgroups in a general hyperbolic group.
However, one cannot hope for statements of the same strength in the general setting.
Rips \cite[Corollary (a)]{rips1982subgroups} constructed hyperbolic groups without f.g.i.p., i.e., examples of finitely-generated subgroups of hyperbolic groups with infinitely-generated intersection. 
The next example shows that one can even take the subgroups to be finitely-generated free groups.

\begin{example}\label{example: two fibers}
    Let $G$ be a hyperbolic free-by-cyclic group with two different fibrations over $\bbZ$ with fibers $H_1,H_2$. 
    That is, $H_1,H_2\normalin G$ are finitely-generated free subgroups such that $G/H_i\simeq \bbZ$. It follows that the intersection $H_1\cap H_2$ is infinitely-generated, since it is a non-trivial, infinite-index, normal subgroup of the free group $H_1$.
\end{example}

The problem in \Cref{example: two fibers} is that the subgroups $H_1,H_2$ are not quasiconvex.
It is well-known that the intersection $H_1\cap H_2$ of quasiconvex subgroups is finitely-generated. 
Moreover, the rank of the intersection can be bounded by a function of the quasiconvexity constants of $H_1, H_2$. 
However, as the following example shows, assuming quasiconvexity is not enough to bound $\rank(H_1\cap H_2)$ as a function of $\rank(H_1),\rank(H_2)$.

\begin{example}\label{example: cyclic cover}
    Let $G=F \rtimes \gen{t}$ be a hyperbolic free-by-cyclic group,
    let $t\in H_1\le G$ be a fixed quasiconvex free group of rank 2. 
    Consider the sequence $H_{2,n} = F\rtimes \gen{t^n}$ of index $n$ subgroups of $G$.
    On the one hand, the subgroups $H_1,H_{2,n}$ are quasiconvex and have bounded rank, namely $\rank(H_1)=2$ and $\rank(H_{2,n})\le \rank(F)+1$ for all $n$. On the other hand, $\rank(H_1\cap H_{2,n}) \to \infty$. 
    More precisely, $\rank(H_1\cap H_{2,n})= n+1$ since $H_1\cap H_{2,n}$ is an index $n$ subgroup of the rank 2 free group $H_1$.
\end{example}

The above examples show we must restrict the class of subgroups we consider.
We prove the following generalization of Howson's Theorem.

\begin{theoremA}\label{main thm}
    Let $G$  be a hyperbolic group. 
    There exists $r=r(r_1,r_2,G)$ such that if $H_1,H_2<G$ are locally quasiconvex, torsion-free subgroups of ranks $r_1,r_2$, then 
    $\rank(H_1\cap H_2)\le r$.
    
\end{theoremA}



Our proof proceeds in two steps, separated into two sections. 
In \Cref{sec:carrier section}, we construct carriers for locally quasiconvex, torsion-free subgroups of a hyperbolic group $G$.
A \emph{carrier} for $H<G$ is a geometric model for $H$, and \Cref{thm: carrier existence} shows the quasiconvexity constant of the carrier depends only on $\rank(H)$ and $G$.
Variations of this result have been obtained in \cite{kapovich2004freely,biringer2025thick,kohav2024ascendingchainsfreequasiconvex}.
We include a statement and proof tailored to our construction.
In \Cref{sec:fiber product section}, we define a notion of coarse fiber product for carriers. 
We show the fiber product of carriers for $H_1,H_2$ carries $H_1\cap H_2$ and that the rank of the fiber product is bounded by a function of the ranks of the factors.
In combination with \Cref{thm: carrier existence}, this proves \Cref{main thm}.

\section{Carriers}\label{sec:carrier section}

For the entirety of this section, we fix a Cayley complex $X$ of a hyperbolic group $G$, i.e. the universal cover of a finite presentation complex.
For $H\le G$, we let $X_H$ be the space $H\backslash X$ with the basepoint $x_H=H\cdot 1$.

\begin{definition}[Carrier]
    A \emph{carrier} of $H$ is a $\pi_1$-surjective graph map $f:(A,*)\to (X_H,x_H)$ where $(A,*)$ is a finite, pointed graph.
    A carrier is \emph{minimal} if $A$ has the minimal number of edges amongst all carriers of $H$.
    A carrier is $\kappa$-\emph{quasigeodesic} if every path $\gamma_1:[0,1]\to X_H$ with endpoints in $f(V(A))$ is homotopic rel endpoints to a path of the form $[0,1]\xrightarrow[]{\cong} [0,\ell]\xrightarrow[]{\gamma_2} A\xrightarrow[]{f} X_H$ where $f\circ \gamma_2$ is a combinatorial path that lifts to a $\kappa$-quasigeodesic in $X$.
\end{definition}

We will often suppress basepoints when discussing carriers.
\smallskip

Clearly, if a quasigeodesic carrier for $H$ exists then $H$ is quasiconvex.
Conversely, we will see in \Cref{stupid quasiconvexity bound}, that if a subgroup is quasiconvex, then any carrier is $\kappa$-quasigeodesic for some $\kappa$ depending on the carrier.
We will prove in \Cref{thm: carrier existence} that for locally quasiconvex, torsion-free $H$ the constant $\kappa$ can be chosen to depend only on $G$ and the rank of $H$. 

\begin{lemma}\label{stupid quasiconvexity bound}
Every carrier graph $f:A\to X_H$ of a quasiconvex subgroup $H$ is $\kappa$-quasigeodesic for some $\kappa=\kappa(H,f,X)$.
\end{lemma}
\begin{proof}
    Let $\hat A\to A$ be the cover corresponding to $\ker(f_*)$, and let $\hat f:\hat A\to X$ be the lift of $\hat A\to A\to X_H$ to $X$. 
    Since $\pi_1A/\ker(f_*)=H$, there is an action of $H$ on $\hat A$ by deck transformations so that $\hat A/H=A$, and the elevation $\hat f$ is $H$-equivariant.
    Since $H$ is quasiconvex in $G$, we get $\hat f$ is a quasiisometric embedding.
    This is equivalent to the carrier $f:A\to X$ being quasigeodesic, and the quasigeodesicity constant clearly depends only on $H$, $f$, and $X$.
\end{proof}

\begin{definition}
    A \emph{topological vertex} is a vertex of degree not equal to two.
    A \emph{topological edge} is a maximal path subgraph such that the internal vertices of the path have degree two.
    For a graph $A$, we denote the set of topological vertices and edges by $\calV A$ and $\calE A$.
\end{definition}

The following lemma relates the rank of a carrier with the degree sum of its topological vertices and the number of topological edges.

\begin{lemma}\label{lem: rank and topological vertex comparison}
    If $A\to X_H$ is a carrier for $H$, then 
    \begin{equation}
    \rank H\leq \sum_{v\in \calV A}\deg(v).
    \end{equation}
    If $A$ has at most a single degree one vertex, then
    \begin{equation}\label{eq: rank vertex and edge comparison}
    |\calV A|\le 2\rank\pi_1 A +1 \quad \text{ and }\quad 
    |\calE A|\le 3\rank \pi_1 A
    \end{equation}
\end{lemma}
\begin{proof}
    The first inequality follows from
    \[\rank H\leq \rank \pi_1A=1-\chi(A)=1-|\calV A|+\frac 12\sum_{v\in \calV A}\deg(v)\leq \sum_{v\in \calV A}\deg(v).\]
    Supposing $A$ has at most a single degree one vertex, we have $\deg(v)\geq 3$ for all other topological vertices.
    Thus
    \begin{align*}
       \rank \pi_1A=1-\chi(A)&=1-|\calV A|+\frac 12\sum_{v\in \calV A}\deg(v) \\
       &\geq \frac13-\frac13\sum_{v\in\calV A}\deg(v)+\frac 12\sum_{v\in \calV A}\deg(v) \\
       &\geq \frac16\sum_{v\in \calV A}\deg(v).
    \end{align*}
    Combining this with 
    $$\sum_{v\in \calV A}\deg(v)\geq 3(|\calV A|-1)+1\quad \text{ and }\quad\sum_{v\in \calV A}\deg v=2|\calE A|,$$ 
    we obtain \eqref{eq: rank vertex and edge comparison}.
\end{proof}

Note that only the basepoint of a minimal carrier can have degree equal one.
As a consequence of \Cref{lem: rank and topological vertex comparison}, if $A\to X_H$ is a minimal carrier with the same rank as $H$, then $\rank H$ is comparable (up to a multiplicative constant) to both $\sum_{v\in \calV A}\deg(v)$ and $|\calE A|$.

\begin{lemma}[Local-to-Global, \cite{coornaertDelzantPapadopoulos1990}~Ch.3 Thm~1.4]
\label{local to global}
    Let $X$ be a $\delta$ hyperbolic graph. For all $\kappa>0$ there exist $\kappa',L$ such that if a path in $X$ is a $(L,\kappa)$-local-quasigeodesic --- that is, the restrictions of the path to subsegments of length $L$ are $\kappa$-quasigeodesic --- then it is $\kappa'$-quasigeodesic.
\end{lemma}

\begin{theoremA}\label{thm: carrier existence}
    For every hyperbolic group $G$, and $r\in \bbN$ there exists $\kappa=\kappa(r,G)$ such that if $H\le G$ is locally quasiconvex, torsion-free then every minimal rank $r$ carrier of $H$ is $\kappa$-quasigeodesic.
\end{theoremA}

\begin{proof}
    Let $f:A\to X_H$ be a minimal carrier for $H$ of rank $r$, let $e_1,\dots,e_t$ be the topological edges of $A$ listed in some order of non-decreasing length, and let $A_i$ be the union $e_1\cup \dots\cup e_i$ for $i=1,\dots, t$.
    By \Cref{lem: rank and topological vertex comparison}, $t$ is bounded by a function of $r$.
    
    Let $\gamma_1:[0,1]\to X_H$ be a path with endpoints on $f(A)$.
    Since $f:A\to X_H$ is a carrier, there exists a shortest combinatorial path $\gamma:[0,\ell]\to A$ such that $\gamma_1$ is homotopic rel endpoints to the path $[0,1]\xrightarrow[]{\cong} [0,\ell]\xrightarrow[]{\gamma} A\to X_H$.
    Let $s$ be the maximal index such that $\gamma$ traverses the edge $e_s$.
    We prove by induction on $s$ that there exists $\kappa(s,G)$ such that $f\circ \gamma$ is $\kappa(s,G)$-quasigeodesic.
    
    \begin{claim}\label{folding lemma}
        For any decomposition $e_i = \beta\cdot \beta' \cdot \beta''$ and any path $\alpha$ in $A_{i-1}$:
        \begin{enumerate}[label = (F\arabic*)]
            \item \label{first folding move} If $\eta$ is a path in $X_H$ homotopic rel endpoints to $f(\alpha\cdot \beta)$ then $|\beta|\le |\eta|$. See \Cref{fig:foldings} left.
            \item \label{second folding move} If $\eta$ is a path in $X_H$ homotopic rel endpoints to $f(\beta''\cdot \alpha\cdot \beta)$ then $|\beta''|\leq |\eta|$ and $|\beta|\le |\eta|$.
            See \Cref{fig:foldings} right.
        \end{enumerate}
    \end{claim}

    \begin{figure}[h]
        \centering
        \includegraphics[]{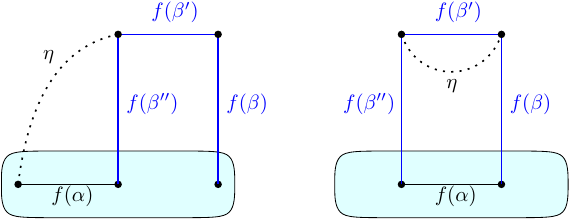}
        \caption{The left and right figures depict the setups of \ref{first folding move} and \ref{second folding move} in \Cref{folding lemma}, respectively. 
        In both figures, the light blue blob is $A_{i-1}$ and the dark blue arc is $f(e)=f(\beta)\cdot f(\beta')\cdot f(\beta'')$.}
        \label{fig:foldings}
    \end{figure}
    \begin{proof}
        \ref{first folding move} Suppose for the sake of contradiction that $|\eta|<|\beta|$.
        Define the graph $A'$ by replacing $e_i$ by an edge $e_i'$ whose endpoints are the endpoints of $\alpha\cdot e_i$. Define $f':A'\to X_H$ by $f'(e_i')=\eta \cdot f(\beta'\beta'')$ and $f'=f$ elsewhere.
        It is easy to see that $(A',f')$ is a shorter carrier of $H$ (of the same rank) in contradiction to the minimality of $(A,f)$.

        \ref{second folding move} Suppose for the sake of contradiction that $|\eta|<|\beta''|$.
        Define the graph $A'$ by replacing $e_i$ by two edges $e_i',e_i''$ 
        such that $e_i'$ has the same initial endpoint as $e_i$, at its terminal endpoint $e_i''$ forms a loop. 
        Define $f':A'\to X_H$ by $f'(e_i')=f(\beta)$, $f'(e_i'')=f(\beta')\eta$ and $f'=f$ elsewhere.
        It is easy to see that $(A',f')$ is a shorter carrier of $H$ in contradiction to the minimality of $(A,f)$.
    \end{proof}

    Write $\gamma = \alpha_1e_s^\pm \alpha_2\dots e_s^\pm \alpha_k$ where $\alpha_j\in A_{s-1}$.
    By induction, each $\alpha_j$ is $\kappa(s-1)$-quasigeodesic.

    We set the following constants:
    \begin{enumerate}
    \item the hyperbolicity constant $\delta$ of $X$,
    \item $\kappa_0 := \kappa(s-1)$,
    \item  by the Morse Lemma, there exists $M_0$ such that any two $\kappa_0$-quasigeodesics with the same endpoints are at Hausdorff distance at most $M_0$,
    \item $\rho_0 := M_0+4\delta$,
    \item there exists $\epsilon=\epsilon(X,G)>0$ so that every infinite-order $g\in G$ has translation length at least $\epsilon$, i.e. $d(w,g^kw)\geq k\epsilon$ for any $w\in X$, 
    \item by the Morse lemma, there exists $\rho_1$ such that if $d(w,gw),d(u,gu)\le 4\delta+6\rho_0$ then $d(w,\{g^n u\}_{n\in\mathbb N})\le \rho_1$,
    \item $\kappa :=\kappa_0+2\rho_0+\rho_1+2$, and
    \item by \Cref{local to global}, there exist $L,\kappa_2$ such that any $(L,\kappa)$-local-quasigeodesic is a $\kappa_2$-quasigeodesic.
    \end{enumerate}
    
    We divide into two cases.

    \textbf{Case 1:} $|e_s|\le L$.
    By local quasiconvexity, the subgroup  $H_s = f_*(\pi_1(A_s))$ is quasiconvex, and we can view $A_i$ as a carrier for $H_s$ by taking the lift $f_s:A_s\to X_{H_s}$ of $f|_{A_s}:A_s\to X_H$ to the cover $X_{H_s}\to X_H$. 
    
    Let $\calB_s$ be the collection of carrier graphs $(B,\ast)\to (X_K,x_K)$ such that the graph $B$ has at most $s$ topological edges, all of length at most $L$.
    The collection $\calB_s$ is finite and depends only on $s$ and $X$:
    Indeed, there are finitely many pointed graphs $(B,\ast)$ with at most $sL$ edges, finitely many groups $K<G$ generated by at most $s$ generators whose length is at most $2sL$, and finitely many based maps $(B,\ast)\to (X_K,x_K)$. 
    Let $\calB^{qc}_s$ be the subcollection of those carrier graphs for which $K$ is quasiconvex. 
    Let $\kappa_1$ be the maximal value of $\kappa(K,f,X)$ of \Cref{stupid quasiconvexity bound} ranging over all carrier graphs in $\calB^{qc}_s$. 
    
    Since $|e_1|\le |e_2|\le \dots \le|e_s|\le L$, the carrier $f_s:A_s\to X_{H_s}$ is in the collection $\calB^{qc}_s$, and thus $\gamma$ is $\kappa_1$-quasigeodesic.

    \textbf{Case 2:} $|e_s|>L$.
    We will prove that $f\circ \gamma$ lifts to a $\kappa_2$-quasigeodesic by showing that it lifts to an $(L,\kappa)$-local-quasigeodesic.
    Let $\gamma'$ be an arbitrary subpath of $\gamma$ of length $\le L$, and let $\tild\gamma'$ be a lift of $f\circ\gamma'$ to $X$ with its endpoints denoted $x,y$. 
    The goal is to prove that 
    \begin{equation}\label{desired distance inequality}
        d(x,y)\ge \tfrac1 {\kappa}|\gamma'| - \kappa.
    \end{equation}
    
    We divide into further subcases.
    
    \textbf{Case 2.1:} $\gamma'=\alpha\cdot \beta$  where $\alpha$ is a path in $A_{i-1}$ and $\beta$ is a subpath of $e_s$. 
    Let $\tild \gamma' = \tild\alpha\cdot  \tild \beta$ where $\tild\alpha, \tild \beta$ are the appropriate lifts of $\alpha, \beta$.
    Let $x,u$ be the endpoints of $\tild \alpha$ and $u,y$ be the endpoints of $\tild \beta$.
    Let $\rho = (x.y)_{u}$ be the Gromov product.
    Let $x',y'$ be the two points on the geodesics $[x,u]$ and $[u,y]$ at distance $\rho$ from $u$. 
    So, $d(x',y')\le 4\delta$.
    Let $x''$ be a point on $\tild\alpha$ such that $d(x'',x')\le M_0$.
    We have $d(x'',y)\le d(x'',x')+d(x',y')+d(y',u) \le M_0+4\delta +(|\beta |-\rho)$.
    By \ref{first folding move}, we have 
    $|\beta|\le d(x'',y)$.
    Therefore, $\rho \le \rho_0=M_0+4\delta$.

    Since $\alpha$ is a $\kappa_0$-quasigeodesic, we have
    \begin{align*}
        d(x,y)\ge d(x,u)+d(u,y) - 2\rho &\ge \tfrac 1{\kappa_0}|\alpha|-\kappa_0+|\beta| -(M_0+4\delta) \\
        &\ge \tfrac 1{\kappa}(|\alpha|+|\beta|) - \kappa \ge \tfrac 1{\kappa} |\gamma'|-\kappa,
    \end{align*} 
    using that $\kappa \ge \max\{\kappa_0+M_0+4\delta,1\}$.
    Thus $\gamma'$ satisfies \ref{desired distance inequality}.

    One similarly handles the case $\gamma'=\beta \cdot \alpha$ where $\alpha$ is a path in $A_{i-1}$ and $\beta$ is a subpath of $e_s$. 
    
    \textbf{Case 2.2:} $\gamma' = \beta'' \cdot \alpha \cdot \beta$ where $\alpha$ is in $A_{i-1}$ and $\beta,\beta''$ are disjoint subsegments of $e_s$.
    We have $\tild\gamma'=\tild\beta'' \cdot \tild\alpha \cdot \tild \beta$.
    Let $x,y$ be the endpoints of $\tild \gamma'$, and let $u,v$ be the endpoints of $\tild \alpha$.

    As in the previous case $(x.v)_{u}\le \rho_0$ and $(u.y)_{v}\le \rho_0$. 
    If $d(u,v)\ge 2\rho_0$, then an argument similar to the previous case shows 
    $d(x,y)\ge \tfrac{1}{\kappa}|\gamma'|-\kappa$.

    Thus we may assume $d(u,v)\le 2\rho_0$.
    By \ref{second folding move}, we have 
    \begin{equation*}
        d(x,y)\ge \tfrac 12(|\beta|+|\beta''|)= \tfrac 12|\gamma'|-\tfrac 12|\alpha|\ge \tfrac 12|\gamma'|-\rho_0\ge \tfrac 1\kappa|\gamma'|-\kappa,
    \end{equation*}
    since $\kappa\ge \max\{2,\rho_0\}$. 
    Hence \ref{desired distance inequality} is satisfied 
    
    \textbf{Case 2.3:} $\gamma' = \beta \cdot \alpha \cdot \beta'$ where $\alpha$ is in $A_{s-1}$ and $\beta,\beta'$ are intersecting subsegments of $e_s$.
    In particular, $\beta,\beta'$ contain the same endpoint of $e_s$.
    As above, we let $\tild\gamma'=\tild\beta \cdot \tild\alpha \cdot \tild\beta'$ and denote the endpoints of $\tild\gamma'$ and $\tild\alpha$ by $x,y$ and $u,v$, respectively.
    
    As before, if $d(u,v)\ge 2\rho_0$ then we are done. 
    Therefore, we may assume $d(u,v)\le 2\rho_0$.
    Let $g\in \pi_1A_{s-1}$ be the element represented by $\alpha$. We have $gu=v$, and so $$d(u,gu)\le 2\rho_0.$$


    Let $w\in \tild\beta$ be such that $d(w,u)=(x.y)_u$.
    Then, there exists $w''\in [u,y]$ such that $d(w,w'')\le \delta$.
    Since the triangle $u,v,y$ is $\delta$-thin, there exists a point $w'\in\tild \beta'$ such that $d(w'',w')\le \delta+2\rho_0$.
    We have $d(w,w')\le 2\delta+2\rho_0$.
    Since $w',gw$ lie on a common geodesic with endpoint $v$, namely the geodesic $\tilde \beta'$ or $g\tilde \beta$, we have
    \begin{multline*}
        d(w',gw)= |d(w',v)-d(gw,v)| = |d(w',v)-d(w,u)|\\
        \le d(w,w')+d(u,v)\le 2\delta +4\rho_0
    \end{multline*}
    and so
    $$d(w,gw)\le d(w,w')+d(w',gw)\le  4\delta+6\rho_0.$$
    
    By the choice of $\rho_1$ there exists $n$ such that $d(w,g^nu)\le \rho_1$.
    As before, \ref{first folding move} implies $(w.g^nu)_u\le \rho_0$, 
    and so
    $$d(w,g^n u) \ge  d(w,u)+d(g^nu,u)-2(w.g^nu)_u\ge d(w,u)-2\rho_0.$$
    Therefore, $d(w,u)\le \rho_1 + 2\rho_0$.

    By the choice of $w$ we get that $(x.y)_u \le \rho_1+2\rho_0$,
    and so 
    $$d(x,y)= d(x,u)+d(u,y)-2(x.y)_u\ge |\beta|+|\beta''|-\rho_1-2\rho_0\geq \frac 1\kappa|\gamma'|-\kappa,$$
    since $\kappa\ge \max\{1,\rho_1+2\rho_0\}$.
    This establishes \eqref{desired distance inequality} in {Case~2.3}.
    
    Setting $\kappa(s) = \max\{\kappa_1,\kappa_2\}$, in both Case 1 and Case 2 we proved that $\gamma$ is a $\kappa(s)$ geodesic. 
    Finally, taking $s=t$ completes the proof of \Cref{thm: carrier existence}.
\end{proof}

\section{Fiber products}\label{sec:fiber product section}
    The goal of the section is to prove \Cref{main thm}. 

\begin{lemma}\label{distance from edges}
    Let $f:A\to X_H$ be a $\kappa$-quasigeodesic carrier.
    For all $\rho\ge 0$ there exists $R=R(\kappa,\rho)$ such that if $e$ is a topological edge of $A$, then $$N_\rho(f(A-e)) \cap f(e) \subseteq f(N_R(\partial e)).$$ 
\end{lemma}

\begin{proof}
    Let $R = \kappa(\rho+\kappa)$. 
    If the inclusion is false, then there exist $x\in Ve$ and $y\in V(A-e)$ such that $d(x,\partial e)>R$ and $d(f(x),f(y))\leq \rho$.
    Let $\gamma_1$ be a geodesic in $X_H$ joining $f(x)$ to $f(y)$. 
    The lift of $\gamma_1$ to $X$ is a geodesic with endpoints $p, q$ that are lifts of $f(x),f(y)$, so $d(p,q)= d(f(x),f(y))\le \rho$.
    We know $\gamma_1$ is homotopic rel endpoints to a path of the form $[0,\ell]\xrightarrow[]{\gamma_2} A\to X_H$ such that $f\circ \gamma_2$ lifts to a $\kappa$-quasigeodesic in $X$ joining $p,q$.
    Since $d(x,\partial e)>R$, we have $R<\ell$. 
    By $\kappa$-quasigeodesicity, we get the contradiction
    $$\rho=\kappa\ii R -\kappa< \kappa\ii \ell - \kappa \le d(p,q)\le \rho. \qedhere$$
\end{proof}

The following lemma says that, up to a controlled increase in rank, one can assume carriers of locally quasiconvex subgroups are embeddings.

\begin{lemma}
\label{lem:embedded carrier}
    For every hyperbolic group $G$, and $r\in \bbN$ there exists $\kappa=\kappa(r,G)$ and $r'=r'(r,G)$ such that if $H\le G$ is a rank $r$ locally quasiconvex subgroup, then $H$ admits a $\kappa$-quasiconvex carrier of rank $r'$ embedded in $X_H$.
\end{lemma}

\begin{proof}
    Let $f:A\to X_H$ be a minimal carrier of rank $r$ for $H$.
    By \Cref{thm: carrier existence}, $f$ is $\kappa$-quasigeodesic for some $\kappa = \kappa(r,G)$.
    Set $A'=f(A)$ and let $f':A'\hookrightarrow X_H$ be the inclusion map.
    Clearly $f'$ is a $\kappa$-quasigeodesic carrier for $H$, so it remains to bound its rank.
    
    We bound $\rank(A')$ by bounding $|\calE A'|$. 
    The topological vertices of $A'$ in $f(e)$ are contained in $f(e)\cap f(A-e)$.
    By \Cref{distance from edges}, there exists $R=R(\kappa,\rho)$ such that $f(e)\cap f(A-e)\subseteq f(N_R(\partial e))$. 
    Consequently, $f(e)$ contains at most $2R+2$ topological vertices and thus at most $2R+1$ topological edges.
    We conclude that $|\calE A'|\le (2R+1)|\calE A|$.
    By \Cref{lem: rank and topological vertex comparison}, $$\rank(A')\le |\calE A'|\le (2R+1)|\calE A|\le 3(2R+1) r. \qedhere$$
\end{proof}

\begin{lemma}\label{nullhomotopic tubes}
    Let $A\hookrightarrow X_H$ be an embedded $\kappa$-quasigeodesic carrier of rank~$r$. 
    For $\rho\ge 0$, there exists $R=R(G,\rho,r)$ such that for all $e\in \calE A$, the map 
    \[\pi_1(N_\rho(e-N_{R}(\partial e))) \to \pi_1(X_H)\]  induced by inclusion is trivial.
\end{lemma}

\begin{proof}
    Let $R=\kappa(2\rho+\kappa)$.
    Let $Y=N_\rho(e-N_{R}(\partial e))$. 
    Let $\tilde e$ be a lift of  $e$ to $X$, and let $\tild Y = N_\rho(\tild e-N_{R}(\partial \tild e))$. 
    We note that the following are equivalent: 
    \begin{itemize}
        \item The map $\pi_1(Y) \to \pi_1(X_H)$ is trivial
        \item $Y$ lifts to $X$.
        \item $\tild Y\cap h\tild Y= \emptyset$ for all nontrivial $h\in H$.
    \end{itemize} 
    Assume for the sake of contradiction that there exists nontrivial $h\in H$ so that $\tild Y\cap h\tild Y\neq \emptyset.$
    Then there exists $x\in \tilde e - N_R(\partial \tilde e)$ and $y\in h(\tilde e - N_R(\partial \tilde e))$ so that $d(x,y)\leq 2\rho$.
    Let $[0,\ell_1]\to X$ be a geodesic joining $x$ to $y$, and let $\gamma_1$ be the composition $[0,\ell_1]\to X\to X_H$.
    By $\kappa$-quasigeodesicity, there exists a path $\gamma_2$ of the form $[0,\ell_2]\to A\to X_H$ such that its lift $\tild \gamma_2$ to $X$ is a $\kappa$-quasigeodesic joining $x$ to $y$.
    Thus we get the following contradiction
    $$2\rho=\kappa\ii R-\kappa < \kappa\ii \ell_2 -\kappa\le d(x,y)\le 2\rho. \qedhere$$
\end{proof}

\begin{proof}[Proof of \Cref{main thm}]
    Let $H_1,H_2$ be locally quasiconvex finitely-generated subgroups of $G$, and let $r_1=\rank (H_1), r_2=\rank(H_2)$. 
    Let $\kappa=\kappa(r_1,r_2,G)$ and $r'=r'(r_1,r_2,G)$ be the constants of \Cref{lem:embedded carrier}, so there exist embedded $\kappa$-quasigeodesic carriers $A_1\subset X_{H_1}, A_2\subset X_{H_2}$ for $H_1,H_2$ of rank at most $r'$.
    We may assume each $A_i$ has at most a single degree one vertex (its basepoint), hence by \Cref{lem: rank and topological vertex comparison}, $$|\calV A_i|\le 2r'+1\quad \text{ and }\quad |\calE A_i|\le 3r'.$$

    By the Morse Lemma, there exists $\rho$ such that any two $\kappa$-quasigeodesics in $X$ with the same endpoints are at Hausdorff distance at most $\rho$.
    Let $\bfA_1 = N_\rho(A_1)$, and let $\bfA_1\otimes A_2$ be the fiber product of the maps $\bfA_1\to X_G$ and $A_2\to X_G$, which are restrictions of the coverings $X_{H_1}\to X_G$ and $X_{H_2}\to X_G$.
    Let $\bfB$ be the component of $\bfA_1\otimes A_2$ containing the basepoint.
    Note that $\bfB$ is naturally a subcomplex of $X_{H_1\cap H_2} \subseteq X_{H_1}\otimes X_{H_2}$.
    
    \begin{claim}\label{bfB carries}
    $\bfB$ is a carrier for $H_1\cap H_2$.
    \end{claim}
    \begin{proof}
    For any closed path $\gamma:I\to X_{H_1\cap H_2}$ based at the basepoint, and for $i=1,2$, there exists a homotopy rel endpoints from $\gamma$ to $\gamma_i$ such that $\gamma_i$ lifts to a $\kappa$-quasigeodesic in $X$ and maps to $A_i$ when composed with the cover $q_i:X_{H_1\cap H_2}\to X_{H_i}$.
    By the Morse Lemma, the quasigeodesics $\tild \gamma_1, \tild \gamma_2$ are at Hausdorff distance at most $\rho$.
    Therefore $\gamma_1,\gamma_2$ are at Hausdorff distance at most $\rho$, and the image of $\gamma_2$ under $q_1$ lies in $\bfA_1$.
    It follows that $\gamma_2$ is a loop in $\bfB$.
    \end{proof}


    Let $R = R(G,\rho,r)$ be the constant of \Cref{nullhomotopic tubes}. 
    For a topological edge $e$ of $A_1$, let $\bfe$ denote $N_\rho(e-N_R(\partial e))\subseteq \bfA_1$.
    Let $\partial \bfe$ denote $N_R(\partial e) \cap \bfe$.
    Let $Y_\bfe$ be a minimal tree in $\bfe$ that contains the vertices of $\partial \bfe$.

    \begin{claim}\label{homotopy to the tree}
        Any path in $\bfe$ with endpoints in $\partial \bfe$ is homotopic rel endpoints in $X_{H_1}$ to a path in $Y_\bfe$.
    \end{claim}
    
    \begin{proof}
        Let $\gamma$ be such a path, and let $\eta$ be a path in $Y_\bfe$ with the same endpoints. The concatenation $\gamma\eta$ is a closed loop in $\bfe$.  By \Cref{nullhomotopic tubes}, $\pi_1(\bfe)\to \pi_1(X_{H_1})$ is trivial, and so $\gamma \eta$ is nullhomotopic in $X_{H_1}$. 
    \end{proof}

    Let $\hat \bfe$ be a component of $q_1\ii(\bfe)$.
    By \Cref{nullhomotopic tubes}, $\hat \bfe$ is a lift of $\bfe$, i.e. the map $q_1|_{\hat \bfe}:\hat \bfe\to \bfe$ is an isomorphism. 
    Let $\partial \hat \bfe$ be the corresponding lift of $\partial \bfe$, i.e. $\partial \hat \bfe=(q_1|_{\hat \bfe})\ii(\partial \bfe)$.
    By \Cref{distance from edges}, the graph $\bfA_1$ decomposes as $\bfA_1 = (\bfA_1-\bfe)\cup _{\partial \bf e}\bfe$.
    Thus the graph $\bfB$ decomposes as $\bfB = (\bfB - \hat \bfe)\cup_{\bfB\cap \partial \hat \bfe} (\bfB\cap \hat \bfe)$.
    We say $\hat \bfe$ is \emph{thin} if $q_2(\hat\bfe\cap \bfB)$ does not contain a topological vertex of $A_2$, and \emph{thick} otherwise.

    Let $B$ be the graph obtained from $\bfB$ by removing $\hat \bfe\cap \bfB$ for all thick components $\hat \bfe$ and replacing them with the lift $Y_{\hat \bfe}$ of the tree $Y_\bfe$, ranging over all topological edges $e$ of $A_1$.


    \begin{claim}
        $B$ is a carrier for $H_1\cap H_2$.
    \end{claim} 
    \begin{proof}
        Let $\gamma: I \to X_{H_1\cap H_2}$ be a loop based at a point of $B$. 
        By \Cref{bfB carries}, we may suppose the image of $\gamma$ lies in $\bfB$. 
        Let $\hat \bfe$ be a thick component of $q_1\ii(\bfe)$ for some $e\in \calE A_1$.
        As we saw, the graph $\bfB$ decomposes as $\bfB = (\bfB - \hat \bfe)\cup_{\bfB\cap \partial \hat \bfe} (\bfB\cap \hat \bfe)$.
        Hence each subpath of $\gamma$ that lies in $\bfB\cap \hat\bfe$ has endpoints in $\bfB\cap \partial \hat \bfe$.
        By \Cref{homotopy to the tree}, such a subpath can be homotoped rel its endpoints into $Y_{\hat \bfe}$.
        Performing this homotopy for each subpath in a thick $\hat\bfe$ results in a loop in $B$.
    \end{proof}

    Finally, we want to show that $\rank(H_1\cap H_2)$ can be bounded in terms of $\rank(H_1),\rank(H_2)$. 
    A vertex $v\in VB$ of degree at least three must project to vertices of degree at least three in $\bfA_1$ and $A_2$.  
    Thus any $v\in VB$ of degree at least three is one of the following types:
        \begin{enumerate}[label = \bf\Roman*.]
        \item $v\in N_\rho(u)\otimes \calV A_2$ for some $u\in \calV A_1$,
        \item $v\in \partial \hat \bfe \cap B$ for some thick lift $\hat \bfe$ of some $e\in \calE A_1$, or
        \item $v\in \calV(Y_{\hat \bfe})$ for some thick lift $\hat \bfe$ of some $e\in \calE A_1$.
        \end{enumerate}

        We let $\mathbf{I}, \mathbf{II}, \mathbf{III}\subset VB$ denote the sets of vertices of degree at least three satisfying the corresponding above condition.
        The vertices of the first type are bounded by
        $$|\bfI|\le C|\calV A_1||\calV A_2|,$$ 
        where $C=C(\rho)$ is the maximal number of vertices in a ball of radius $\rho$ in $X_{H_1}$. 
        
        The number of thick lifts $\hat \bfe$ is bounded by $|\calE A_1||\calV A_2|$.
        For any thick lift $\hat \bfe$, the number of vertices in $\partial \hat \bfe \cap B$ can be bounded by $2C'$ where $C'=C'(R)$ is the maximal number of vertices in a ball of radius $R$ in $X_{H_1}$.
        Thus we have 
        $$|\mathbf{II}|\le 2C'|\calE A_1||\calV A_2|.$$
        
        The leaves of $Y_\bfe$ are contained in $\partial \bfe$, so the number of leaves is at most $|\partial \bfe| \le 2C'$. 
        The number of vertices of degree at least three in a (finite) tree is bounded by the number of leaves, so there are at most $2C'$ vertices of degree at least three in $Y_{\hat \bfe}$. 
        Thus $Y_\bfe$ has at most $4C'$ topological vertices, and the third type vertices are bounded by
        $$|\mathbf{III}|\le 4C'|\calE A_1||\calV A_2|.$$

        Let $B'\subset B$ be the core subgraph of $B$, i.e. $\pi_1B'\to \pi_1B$ is an isomorphism, and $B'$ has no degree one vertices except maybe its basepoint.
        Note that $\calV B'$ is a subset of the vertices of degree at least three in $B$ and maybe the basepoint.
        So, summing our above bounds, we have
        $$|\calV B'|\le |\bfI|+|\mathbf{II}|+|\mathbf{III}|\le C|\calV A_1||\calV A_2| + 6C'|\calE A_1||\calV A_2|.$$
        
        Since $B'\subseteq X_{H_1\cap H_2}$ its degrees are bounded by $d:=2\rank(G)$. 
       Applying both parts of \Cref{lem: rank and topological vertex comparison}, we have
        \begin{align*}
            \rank(H_1\cap H_2) &\le 2d|\calV B'| \\
            &\le 2d(C|\calV A_1||\calV A_2| + 6C'|\calE A_1||\calV A_2|)\\
            &\le 2d(C(2r'+1)^2 + 6C'(3r')(2r'+1)).
        \end{align*}
        This inequality completes the proof since the quantities in the last expression depend only on $G$ and the ranks of $H_1,H_2$.
\end{proof}


We end the article with two questions:

\begin{question}
    Does \Cref{main thm} remain true when the subgroups have torsion?
\end{question}

Note that \Cref{thm: carrier existence} fails as stated if we allow torsion. 

\begin{question}
    Does there exist $r=r(r_1,r_2,d)$ such that for any locally quasiconvex subgroups $H_1,H_2\le \textrm{Isom}(\bbH^d)$ of ranks $r_1,r_2$, the rank of $H_1\cap H_2$ is at most $r$?
\end{question}
Note that if $H_1,H_2$ are contained in a lattice $\Gamma\le \textrm{Isom}(\bbH^d)$ then \Cref{main thm} shows that $r$ can be taken as a function of $r_1,r_2$ and $\Gamma$.

\bibliographystyle{plain}
\bibliography{biblio}
\end{document}

%% file: preamble.tex
\usepackage[leqno]{amsmath}
\usepackage{amssymb,amsfonts,xfrac,MnSymbol}
\usepackage{amsthm}
\usepackage{thmtools}
\usepackage{cite}
\usepackage{hyperref}
\usepackage{todonotes}
\usepackage{enumitem}
\usepackage{graphicx}
\usepackage{tikz-cd, tikz}
\usepackage{color}
\usepackage{import}
\usepackage{cleveref}

\numberwithin{equation}{section}

\newtheorem{theorem}{Theorem}[section]
\newtheorem{claim}[theorem]{Claim}

\newtheorem{lemma}[theorem]{Lemma}

\newtheorem*{theorem*}{Theorem}
\newtheorem*{claim*}{Claim}
\newtheorem*{proposition*}{Proposition}
\newtheorem*{lemma*}{Lemma}
\newtheorem*{corollary*}{Corollary}

\newtheorem{theoremA}{Theorem}

\theoremstyle{definition}
\newtheorem{definition}[theorem]{Definition}

\newtheorem{example}[theorem]{Example}
\newtheorem{question}[theorem]{Question}

\newtheorem*{definition*}{Definition}
\newtheorem*{observation*}{Observation}
\newtheorem*{remark*}{Remark}
\newtheorem*{example*}{Example}
\newtheorem*{question*}{Question}
\newtheorem*{exercise*}{Exercise}
\newtheorem*{fact*}{Fact}
\newtheorem*{notation*}{Notation}

\newcommand{\bbH}{\mathbb{H}}

\newcommand{\bbN}{\mathbb{N}}

\newcommand{\bbZ}{\mathbb{Z}}

\newcommand{\bfA}{\mathbf{A}}
\newcommand{\bfB}{\mathbf{B}}

\newcommand{\bfI}{\mathbf{I}}

\newcommand{\bfe}{\mathbf{e}}

\newcommand{\calB}{\mathcal{B}}

\newcommand{\calE}{\mathcal{E}}

\newcommand{\calV}{\mathcal{V}}

\newcommand{\normalin}{\lhd}

\newcommand{\ii}{^{-1}}

\newcommand{\gen}[1]{\left< #1 \right>}

\newcommand{\tild}[1]{\widetilde{#1}}

\DeclareMathOperator{\rank}{rank}

%% file: biblio.bib
@article{antolin2022hanna,
  title={The Hanna Neumann conjecture for surface groups},
  author={Antol{\'\i}n, Yago and Jaikin-Zapirain, Andrei},
  journal={Compositio Mathematica},
  volume={158},
  number={9},
  pages={1850--1877},
  year={2022},
  publisher={London Mathematical Society}
}

@book{biringer2025thick,
  title={Thick hyperbolic 3-manifolds with bounded rank},
  author={Biringer, Ian and Souto, Juan},
  volume={316},
  number={1607},
  year={2025},
  publisher={American Mathematical Society}
}

@book {coornaertDelzantPapadopoulos1990,
    AUTHOR = {Coornaert, M. and Delzant, T. and Papadopoulos, A.},
     TITLE = {G\'eom\'etrie et th\'eorie des groupes},
    SERIES = {Lecture Notes in Mathematics},
    VOLUME = {1441},
      NOTE = {Les groupes hyperboliques de Gromov. [Gromov hyperbolic
              groups],
              With an English summary},
 PUBLISHER = {Springer-Verlag, Berlin},
      YEAR = {1990},
     PAGES = {x+165},
      ISBN = {3-540-52977-2},
   MRCLASS = {57M07 (20F32)},
  MRNUMBER = {1075994},
MRREVIEWER = {John\ Meier},
}

@article {friedman2015sheaves,
    AUTHOR = {Friedman, Joel},
     TITLE = {Sheaves on graphs, their homological invariants, and a proof
              of the {H}anna {N}eumann conjecture: with an appendix by
              {W}arren {D}icks},
      NOTE = {With an appendix by Warren Dicks},
   JOURNAL = {Mem. Amer. Math. Soc.},
  FJOURNAL = {Memoirs of the American Mathematical Society},
    VOLUME = {233},
      YEAR = {2015},
    NUMBER = {1100},
     PAGES = {xii+106},
      ISSN = {0065-9266,1947-6221},
      ISBN = {978-1-4704-0988-3},
   MRCLASS = {20E05 (05C10 55N30)},
  MRNUMBER = {3289057},
MRREVIEWER = {Matthew\ C. B. Zaremsky},
       DOI = {10.1090/memo/1100},
       URL = {https://doi.org/10.1090/memo/1100},
}

@article {howson1954intersection,
    AUTHOR = {Howson, A. G.},
     TITLE = {On the intersection of finitely generated free groups},
   JOURNAL = {J. London Math. Soc.},
  FJOURNAL = {The Journal of the London Mathematical Society},
    VOLUME = {29},
      YEAR = {1954},
     PAGES = {428--434},
      ISSN = {0024-6107,1469-7750},
   MRCLASS = {20.0X},
  MRNUMBER = {65557},
MRREVIEWER = {Don\ Higman},
       DOI = {10.1112/jlms/s1-29.4.428},
       URL = {https://doi.org/10.1112/jlms/s1-29.4.428},
}

@article{kapovich2004freely,
  title={Freely indecomposable groups acting on hyperbolic spaces},
  author={Kapovich, Ilya and Weidmann, Richard},
  journal={International Journal of Algebra and Computation},
  volume={14},
  number={02},
  pages={115--171},
  year={2004},
  publisher={World Scientific}
}

@misc{kohav2024ascendingchainsfreequasiconvex,
      title={Ascending Chains of Free Quasiconvex Subgroups}, 
      author={Jack Kohav and Nir Lazarovich},
      year={2024},
      eprint={2405.13534},
      archivePrefix={arXiv},
      primaryClass={math.GR},
      url={https://arxiv.org/abs/2405.13534}, 
}

@article {mineyev2012submultiplicativity,
    AUTHOR = {Mineyev, Igor},
     TITLE = {Submultiplicativity and the {H}anna {N}eumann conjecture},
   JOURNAL = {Ann. of Math. (2)},
  FJOURNAL = {Annals of Mathematics. Second Series},
    VOLUME = {175},
      YEAR = {2012},
    NUMBER = {1},
     PAGES = {393--414},
      ISSN = {0003-486X,1939-8980},
   MRCLASS = {57M07 (20E05)},
  MRNUMBER = {2874647},
MRREVIEWER = {Peter\ A.\ Linnell},
       DOI = {10.4007/annals.2012.175.1.11},
       URL = {https://doi.org/10.4007/annals.2012.175.1.11},
}

@article {neumann1956hanna,
    AUTHOR = {Neumann, Hanna},
     TITLE = {On the intersection of finitely generated free groups},
   JOURNAL = {Publ. Math. Debrecen},
  FJOURNAL = {Publicationes Mathematicae Debrecen},
    VOLUME = {4},
      YEAR = {1956},
     PAGES = {186--189},
      ISSN = {0033-3883,2064-2849},
   MRCLASS = {20.0X},
  MRNUMBER = {78992},
MRREVIEWER = {Don\ Higman},
       DOI = {10.5486/pmd.1956.4.3-4.12},
       URL = {https://doi.org/10.5486/pmd.1956.4.3-4.12},
}

@incollection {neumann1990walter,
    AUTHOR = {Neumann, Walter D.},
     TITLE = {On intersections of finitely generated subgroups of free
              groups},
 BOOKTITLE = {Groups---{C}anberra 1989},
    SERIES = {Lecture Notes in Math.},
    VOLUME = {1456},
     PAGES = {161--170},
 PUBLISHER = {Springer, Berlin},
      YEAR = {1990},
      ISBN = {3-540-53475-X},
   MRCLASS = {20E05 (20E07)},
  MRNUMBER = {1092229},
MRREVIEWER = {Julian\ Petresco},
       DOI = {10.1007/BFb0100737},
       URL = {https://doi.org/10.1007/BFb0100737},
}

@article{rips1982subgroups,
  title={Subgroups of small cancellation groups},
  author={Rips, Eliyahu},
  journal={Bulletin of the London Mathematical Society},
  volume={14},
  number={1},
  pages={45--47},
  year={1982},
  publisher={Wiley Online Library}
}

@article{soma1990intersection,
  title={Intersection of finitely generated surface groups},
  author={Soma, Teruhiko},
  journal={Journal of Pure and Applied Algebra},
  volume={66},
  number={1},
  pages={81--95},
  year={1990},
  publisher={Elsevier}
}

@article{soma1991intersection,
  title={Intersection of Finitely Generated Surface Groups {II}},
  author={Soma, Teruhiko},
  journal={Bulletin of the Kyushu Institute of Technology. Mathematics, natural science},
  number={38},
  pages={13--22},
  year={1991},
  publisher={九州工業大学工学部}
}
